\documentclass[10pt, reqno]{amsart}

\usepackage{amsmath,amssymb,amsthm,amsfonts,verbatim}
\usepackage{microtype}
\usepackage[all,2cell]{xy}
\usepackage{mathtools}
\usepackage{graphicx}
\usepackage{hyperref}
\usepackage{mathrsfs}
\usepackage{color}
\usepackage{enumerate}

\CompileMatrices

\usepackage[top=1.3in,bottom=1.9in,left=1.4in,right=1.4in]{geometry}

\theoremstyle{plain}
\newtheorem{theorem}{Theorem}[section]
\newtheorem{maintheorem}{Theorem}

\newtheorem{proposition}[theorem]{Proposition}
\newtheorem{lemma}[theorem]{Lemma}

\theoremstyle{definition}
\newtheorem{definition}[theorem]{Definition}

\newtheorem{remark}[theorem]{Remark}

\newcommand{\nc}{\newcommand}
\nc{\dmo}{\DeclareMathOperator}

\nc{\Q}{\mathbb{Q}}
\nc{\F}{\mathbb{F}}
\nc{\R}{\mathbb{R}}
\nc{\Z}{\mathbb{Z}}
\nc{\C}{\mathbb{C}}
\nc{\Ell}{\mathcal{L}}
\nc{\M}{\mathcal{M}}
\nc{\K}{\mathcal{K}}
\nc{\I}{\mathcal{I}}
\nc{\disk}{\mathbb{D}}
\nc{\hyp}{\mathbb{H}}

\nc{\CP}{\mathbb{CP}}
\nc{\cS}{\mathcal{S}}
\dmo{\Mod}{Mod}
\dmo{\Diff}{Diff}
\dmo{\Homeo}{Homeo}
\dmo{\dist}{dist}
\dmo\BDiff{BDiff}
\dmo\SO{SO}
\dmo\Hom{Hom}
\dmo\SL{SL}
\dmo\Sp{Sp}
\dmo\rank{rank}
\dmo\sig{sig}
\dmo\Out{Out}
\dmo\Aut{Aut}
\dmo\Inn{Inn}
\dmo\GL{GL}
\dmo\PSL{PSL}
\dmo\BHomeo{BHomeo}
\dmo\EHomeo{EHomeo}
\dmo\EDiff{EDiff}
\nc\Sig{\Sigma}
\dmo\Teich{Teich}
\dmo\Fix{Fix}
\nc{\pair}[1]{\langle #1 \rangle}
\nc{\abs}[1]{\left| #1 \right|}
\nc{\action}{\circlearrowright}
\nc{\norm}[1]{\left | \left | #1 \right | \right |}
\nc{\abcd}[4]{\left(\begin{array}{cc} #1 & #2 \\ #3 & #4 \end{array}\right)}
\nc{\into}\hookrightarrow
\dmo{\Isom}{Isom}
\nc{\normal}{\vartriangleleft}
\dmo{\Vol}{Vol}
\dmo{\im}{Im}
\dmo{\Push}{Push}
\dmo{\Conf}{Conf}
\dmo{\PConf}{PConf}
\dmo{\id}{id}
\dmo{\Jac}{Jac}

\renewcommand{\epsilon}{\varepsilon}
\nc{\coloneq}{\mathrel{\mathop:}\mkern-1.2mu=}
\nc{\margin}[1]{\marginpar{\scriptsize #1}}
\nc{\para}[1]{\medskip\noindent\textbf{#1.}}
\nc{\red}[1]{\textcolor{red}{#1}}

\title[Cup products, Johnson invariants, and MMM classes]{Cup products in surface bundles, higher Johnson invariants, and MMM classes}

\author{Nick Salter}
\email{nks@math.uchicago.edu}
\date{\today}

\address{Department of Mathematics\\ University of Chicago\\ 5734 S. University Ave., Chicago, IL 60637}

\begin{document}
\maketitle
\begin{abstract}
In this paper we prove a family of results connecting the problem of computing cup products in surface bundles to various other objects that appear in the theory of the cohomology of the mapping class group $\operatorname{Mod}_g$ and the Torelli group $\mathcal{I}_g$. We show that N. Kawazumi's twisted MMM class $m_{0,k}$ can be used to compute $k$-fold cup products in surface bundles, and that $m_{0,k}$ provides an extension of the higher Johnson invariant $\tau_{k-2}$ to $H^{k-2}(\operatorname{Mod}_{g,*}, \wedge^k H_1)$. These results are used to show that the behavior of the restriction of the even MMM classes $e_{2i}$ to $H^{4i}(\mathcal{I}_g^1)$ is completely determined by $\im(\tau_{4i}) \le \wedge^{4i+2}H_1$, and to give a partial answer to a question of D. Johnson. We also use these ideas to show that all surface bundles with monodromy in the Johnson kernel $\mathcal K_{g,*}$ have cohomology rings isomorphic to that of a trivial bundle, implying the vanishing of all $\tau_i$ when restricted to $\mathcal K_{g,*}$. 
\end{abstract}

\section{Introduction}

The theme of this paper is the central role that the structure of the cup product in surface bundles plays in the understanding of the cohomology of the mapping class group and its subgroups. We use this perspective to gain a new understanding of the relationships between several well-known cohomology classes, and we also use these ideas to study the topology of surface bundles.

Denote by $\Mod_g$ (resp. $\Mod_{g,*}, \Mod_g^1$) the mapping class group of a closed oriented surface of genus $g$ (resp. of a closed oriented surface with a marked point, of a surface with one boundary component). The {\em Torelli group} $\I_g$ is defined as the kernel of the symplectic representation $\Psi: \Mod_g \to \Sp(2g,\Z)$; there are analogous definitions of $\I_{g,*}$ and $\I_g^1$. When left unspecified, all homology and cohomology groups will be taken to have coefficients in $\Q$. In particular, we will use the abbreviations $H_1 := H_1(\Sigma_g;\Q)$ and $H^1 := H^1(\Sigma_g;\Q)$.

For $i \ge 1$, there is a class $e_i \in H^{2i}(\Mod_g)$ known as the $i^{th}$ {\em Mumford-Morita-Miller class} (hereafter abbreviated to MMM class). See Definition \ref{definition:mij}. The Madsen-Weiss theorem \cite{MW} asserts that the so-called ``stable'' rational cohomology of $\Mod_g$ is generated by the MMM classes, and apart from a few sporadic low-genus examples, the algebra generated by the classes $e_i$ are the only known elements of $H^*(\Mod_g)$. In \cite{kawazumi}, N. Kawazumi introduced a generalization of the MMM classes, defining classes $m_{ij} \in H^{2i + j - 2}(\Mod_{g,*}; H_1^{\otimes j})$, specializing to $m_{i,0} =e_{i-1}$. Again, see Definition \ref{definition:mij}. 

The content of Theorem \ref{theorem:A} below is that the cup product form on the total space $E$ gives a characteristic class for surface bundles. Theorem \ref{theorem:A} also gives an ``intrinsic meaning'' to the twisted MMM class $m_{0,k}$ in much the same way that the first MMM class $e_1 \in H^2(\Mod_g)$ has an interpretation as the signature of the total space of a surface bundle over a surface (see \cite[Proposition 4.11]{moritabook}).

\begin{maintheorem}[Cup product as characteristic class]\label{theorem:A}
For all $k\ge 2$ and $g \ge 2$, the twisted MMM class $m_{0,k} \in H^{k-2}(\Mod_{g,*}; \wedge^k H_1)$ computes the cup product in surface bundles in the following sense:

Suppose $B$ is a paracompact Hausdorff space and $f: B \to K(\Mod_{g,*},1)$ is a map classifying a surface-bundle-with-section $\pi: E \to B$. Then for all $i \ge 0$ there is a splitting of vector spaces
\begin{equation}\label{equation:splitting}
H^i(E) \cong H^{i-2}(B)  \oplus H^{i-1}(B; H_1) \oplus H^i(B).
\end{equation}

Let $\epsilon: H^{i-1}(B; H_1) \to H^i(E)$ denote the inclusion associated to this splitting. For $1 \le i \le k$ and any $d_1, \dots, d_k$, let $x_i \in H^{d_i-1}(B; H_1)$; for convenience set $D := \sum d_i$. Then there are the following expressions for the components of the product $\epsilon(x_1) \dots \epsilon(x_k) \in H^D(E)$ in the splitting (\ref{equation:splitting}):
\begin{align}
\label{equation:comp1}H^{D-2}(B)\mbox{- component:} & \quad & (-1)^\gamma m_{0,k}\lrcorner(x_1, \dots , x_k)\\
\label{equation:comp2}H^{D-1}(B; H_1)\mbox{- component:} & \quad & (-1)^{\gamma + 1}\epsilon(m_{0,k+1}\lrcorner(x_1, \dots, x_k))\\
\label{equation:comp3}H^D(B)\mbox{- component:} 	& \quad & 0
\end{align}
(see Definition \ref{definition:intprod} for the meaning of $m_{0,j}\lrcorner(x_1,  \dots,  x_k)$, and Equation (\ref{equation:gamma}) for the definition of $\gamma$).
\end{maintheorem}

The line of thought culminating in Theorem \ref{theorem:A} begins with D. Sullivan \cite{sullivan}, who showed that every element of $\wedge^3 V$ (for $V$ an arbitrary finitely generated torsion-free $\Z$-module) arises as the cup product form $\wedge^3 H^1(M; \Z) \to \Z$ for some $3$-manifold $M$. Johnson \cite{johnsonhom} incorporated some of these ideas in his far-reaching theory of the Johnson homomorphism $\tau: H_1(\I_{g,*}) \to \wedge^3 H_1$, one definition of which is by means of the cup product form in a $3$-manifold fibering as a surface bundle over $S^1$. 

In one direction, the Johnson homomorphism was generalized by S. Morita \cite{moritaextension}, who constructed an extension of $\tau$ by means of a class $\tilde k \in H^1(\Mod_{g,*}; \wedge^3 H_1)$ restricting to $\tau$ on $\I_{g,*}$. In \cite{moritalinear96}, he showed that all of the MMM classes $e_i$ can be expressed in terms of $\tilde k$.

Another generalization of the Johnson homomorphism was given by Johnson himself \cite{johnsonsurvey}, who gave a definition of ``higher Johnson invariants'' $\tau_k: H_k(\I_{g,*}) \to \wedge^{k+2}H_1$ (see Definition \ref{definition:tauk}), but his definition was formulated as a generalization of a different method of constructing the Johnson homomorphism. Theorem \ref{theorem:B} below can be viewed as a synthesis of Morita's and Johnson's perspectives, in that it shows that the twisted MMM classes $m_{0,k}$ restrict on $\I_{g,*}$ to (a multiple of) $\tau_{k-2}$.

\begin{maintheorem}[Extending the higher Johnson invariants]\label{theorem:B}
There is an equality for all $g \ge 2$ and $k \ge 2$
\[
m_{0,k} = (-1)^{k} k!\ \tau_{k-2}
\]
as elements of $H^{k-2}(\I_{g,*}; \wedge^k H_1) \cong \Hom(H_{k-2}(\I_{g,*}; \Q), \wedge^k H_1)$. 
\end{maintheorem}

The cases $k = 2,3$ were established by Kawazumi and Morita in \cite{km}. In \cite{churchfarb}, T. Church and B. Farb developed a method for studying the map $\tau_k$. A central component of their computation is the principle that, when viewed as a homomorphism $H_k(\I_{g,*}) \to \wedge^{k+2} H_1$, the Johnson invariant $\tau_k$ is a map of representations of $\Sp(2g, \Q)$. Johnson showed in \cite{johnsonhom} that $\tau = \tau_1$ is a rational isomorphism and in \cite[Question C]{johnsonsurvey}, asked if the same was true for all $\tau_k$. In \cite{hain}, R. Hain showed that $\tau_2$ was not injective. Church and Farb later used their methods to show that $\tau_k$ is not injective for any $2 \le k < g$. This leaves the question of surjectivity of $\tau_k$ as an unresolved aspect of the theory of the cohomology of $\I_{g,*}$. Church and Farb showed that $\tau_2: H_2(\I_{g,*}) \to \wedge^4 H_1$ is a surjection, but did not address higher $k$, or the behavior of $\tau_2$ on $\I_g^1$. 

In the following theorem, we show that the question of surjectivity of $\tau_k$ (when pulled back to $\I_{g}^1$) is intimately related to another well-known open question about the homology of the Torelli group. It is well-known (see, e.g. the introduction to \cite{moritalinear96}) that the MMM classes $e_{2i+1}$ of {\it odd} index vanish when restricted to $\I_g$. However, the behavior of the {\it even}-index classes $e_{2i}$ on $\I_g$ is completely unknown. 

\begin{maintheorem}[Higher Johnson invariants detect MMM classes]\label{theorem:C}
For all $i$, the restriction of $e_{i}$ to $H^{2i}(\I_g^1; \Q)$ is nonzero if and only if the $\Sp(2g, \Q)$-representation $\im(\tau_{2i}: H_{2i}(\I_g^1) \to \wedge^{2i+2}H_1)$ contains a copy of the trivial representation $V(\lambda_0)$. 
\end{maintheorem}

The primary case of interest is of course $i$ even, but as a corollary of Theorem \ref{theorem:C} and the vanishing of $e_{2i-1}$ on $\I_g^1$, it follows that for all $i \ge 1$, the map $\tau_{4i-2}: H_{4i-2}(\I_g^1) \to \wedge^{4i}H_1$ fails to contain a copy of $V(\lambda_0)$, even though $\wedge^{4i}H_1$ always does. This gives a partial resolution of Johnson's question.

\begin{maintheorem}[Non-surjectivity of $\tau_{4k-2}$]\label{theorem:nonsurj}
For all $k \ge 1$, the map
\[
\tau_{4k-2}: H_{4k-2}(\I_g^1) \to \wedge^{4k} H_1
\]
is not surjective.
\end{maintheorem}

As an application of Theorems \ref{theorem:A} and \ref{theorem:B}, we obtain some results concerning the topology of surface bundles. If $\pi: E \to B$ is a $\Sigma_g$-bundle with monodromy contained in $\I_g$, it is well-known that $H_*(E) \cong H_*(B) \otimes H_*(\Sigma_g)$, an isomorphism of graded vector spaces (see Section \ref{subsection:SB} for the relevant terminology). Briefly put, surface bundles with Torelli monodromy are ``homology products''. In general, the additive isomorphism $H^*(E) \cong H^*(B) \otimes H^*(\Sigma_g)$ is very far from being an isomorphism of rings, as the rich theory of the Johnson homomorphism attests to.  For individual elements $\phi \in \Mod_g$, it is well-understood when a mapping torus $\pi: M_\phi^3 \to S^1$ satisfies a multiplicative isomorphism $H^*(M_\phi) \cong H^*(S^1) \otimes H^*(\Sigma_g)$: such an isomorphism holds if and only if $\phi \in \K_g$, the so-called {\em Johnson kernel} (see the beginning of Section \ref{section:applications} and in particular (\ref{equation:johnsonkernel})). However, if $\pi: E \to B$ is a $\Sigma_g$-bundle over a higher-dimensional $B$ with monodromy contained in $\mathcal K_g$, it is not {\em a priori} clear whether a multiplicative isomorphism $H^*(E) \cong H^*(B) \otimes H^*(\Sigma_g)$ must hold. We show that this is the case.

\begin{maintheorem}[K\"unneth formula]\label{theorem:jkbundles}
Let $\pi: E \to B$ be a $\Sigma_g$-bundle over a paracompact Hausdorff space $B$ with monodromy $\rho: \pi_1 B \to \K_{g,*}$ contained in the Johnson kernel. Then there is an isomorphism of rings
\[
H^*(E) \cong H^*(B) \otimes H^*(\Sigma_g).
\]
\end{maintheorem}

The case $B = S^1$ is essentially a {\em definition} of $\K_{g,*}$. The case $B = \Sigma_h$ a surface was shown by the author in \cite{salter} by giving an explicit construction of a basis of cycles suitable for computing the intersection product in homology; this was applied to the problem of counting the number of distinct surface bundle structures on $4$-manifolds.\\

A final corollary of this theorem is the vanishing of {\em all} higher Johnson invariants on $\K_{g,*}$. As remarked above, the vanishing of $\tau = \tau_1$ on $\K_{g,*}$ is a definition, but it is not {\em a priori} clear that this implies the vanishing of higher invariants. Nonetheless, the results of the paper combine to show that this is the case.

\begin{maintheorem}[Vanishing of $\tau_k$ on $\K_{g,*}$]\label{theorem:jkvanish}
For each $k \ge 1$, the restriction of $\tau_k \in H^k(\I_{g,*}, \wedge^{k+2}H_1)$ to $\K_{g,*}$ is zero.
\end{maintheorem}

The methods of the paper are primarily homological and make heavy use of the theory of the Gysin homomorphism. As the central objects of study are the twisted MMM classes $m_{ij}$ introduced by Kawazumi in \cite{kawazumi}, we will frequently make reference to their theory, especially some later developments by Kawazumi-Morita in \cite{km}.\\

In Section \ref{section:prelim}, we review some preliminary material, including the relationship between surface bundles and the mapping class group, some constructions from multilinear algebra and symplectic representation theory, and a primer on the Gysin homomorphism. Section \ref{section:MMM} is a primer on Kawazumi and Morita's work on the twisted MMM classes. The latter four sections are devoted to the proofs of theorems \ref{theorem:A}, \ref{theorem:B}, \ref{theorem:C}, \ref{theorem:jkbundles} respectively. 

\para{Acknowledgements} Many thanks are due to Madhav Nori, for the inspiring conversations that sparked my interest in and approach to this problem. I would also like to thank Ilya Grigoriev and Aaron Silberstein for helpful discussions along the way. As always, this paper would not have been possible without continued interest, support, and guidance from Benson Farb, as well as many comments on preliminary drafts.

\section{Preliminaries}\label{section:prelim}
\subsection{Surface bundles and the mapping class group}\label{subsection:SB}
A {\em surface bundle} is a fiber bundle $\pi: E \to B$ with fibers $\pi^{-1}(b) \cong \Sigma_g$ for some $g$; in this paper, $g\ge 2$. A {\em section} of a surface bundle $\pi:E \to B$ is a map $\sigma: B \to E$ satisfying $\pi \circ \sigma = \id$. The {\em monodromy representation} associated to $\pi: E \to B$ is the homomorphism $\rho: \pi_1B \to \Mod_g$ that records the isotopy class of the diffeomorphisms of the fiber obtained by parallel transport around loops in $B$. When $\pi: E \to B$ is equipped with a section, $\rho$ lifts to a homomorphism $\rho: \pi_1 B \to \Mod_{g,*}$. 

There is a classifying space $\BDiff(\Sigma_g)$ (resp. $\BDiff(\Sigma_g, *)$) for surface bundles (resp. for surface bundles equipped with a section). A fundamental theorem of Earle-Eells \cite{ee}, in combination with some basic algebraic topology, implies that there are homotopy equivalences
\begin{align*}
\BDiff(\Sigma_g) 	&\simeq K(\Mod_g,1)\\
\BDiff(\Sigma_g, *)	&\simeq K(\Mod_{g,*},1).
\end{align*}
This implies that, given a group extension 
\begin{equation}\label{equation:gextension}
1 \to \pi_1(\Sigma_g) \to \Pi_* \to \Pi \to 1,
\end{equation}
there is an associated $\Sigma_g$-bundle $\pi: K(\Pi_*,1) \to K(\Pi, 1)$ for which the monodromy representation $\rho: \Pi \to \Mod(\Sigma_g)$ coincides with the map $\Pi \to \Out(\pi_1(\Sigma_g)) \cong \Mod(\Sigma_g)$ attached to the group extension (\ref{equation:gextension}). The extension (\ref{equation:gextension}) splits if and only if $\rho$ lifts to $\rho: \Pi \to \Aut(\pi_1 \Sigma_g) \cong \Mod_{g,*}$.  Because of this equivalence, we will be somewhat lax in passing between the setting of surface bundles and the setting of group extensions with surface group kernel.

In light of the homotopy equivalences above, one can interpret elements of $H^*(\Mod_g; M)$ (for an arbitrary $\Q \Mod_g$-module $M$) as ``$M$-valued characteristic classes of $\Sigma_g$-bundles''.

\subsection{Symplectic multilinear algebra} \label{subsection:multi}
In this subsection, we lay out some basic facts concerning multilinear algebra over the $\Q$-vector space $H_1(\Sigma_g; \Q)$, as well as the representation theory of the symplectic group. 

We recall the definitions $H_1 := H_1(\Sigma_g; \Q)$ and $H^1 := H^1(\Sigma_g; \Q)$. The intersection pairing furnishes a nondegenerate alternating $\Sp(2g, \Q)$-invariant form $\mu: H_1^\otimes 2 \to \Q$. This form extends to a nondegenerate pairing
$C_{k}: (H_1^{\otimes k})^{\otimes 2} \to \Q$ given by
\begin{equation}\label{equation:Cdef}
(a_1 \otimes \dots \otimes a_k )\otimes (b_1 \otimes \dots \otimes b_k) \mapsto \prod_{i = 1}^k \mu(a_i \otimes b_i).
\end{equation}
For $u, v \in H_1^{\otimes k}$, the pairing satisfies $C_k(u \otimes v) = (-1)^k C_k(v \otimes u)$.  

By convention, given a vector space $V$, the $k^{th}$ exterior power $\wedge^k V$ will always be defined as a {\em quotient} of $V^{\otimes k}$ by imposing the skew-symmetry relations. Define the projection $q: V^{\otimes k} \to \wedge^k V$. There is a lift $L: \wedge^k V \to V^{\otimes k}$ given by
\begin{equation}\label{equation:L}
L(a_1 \wedge \dots \wedge a_k) = \sum_{\tau \in S_k} (-1)^{\tau} a_{\tau(1)} \otimes \dots \otimes a_{\tau(k)}
\end{equation}
(to lighten the notational load, we will omit reference to $k$, which should be clear from context). By construction, $q \circ L = k! \id$. 

There is a natural pairing $C'_k: (\wedge^k H_1)^{\otimes 2} \to \Q$ given by
\begin{equation}\label{equation:C'def}
(a_1 \wedge \dots \wedge a_k) \otimes (b_1 \wedge \dots \wedge b_k) \mapsto \det(\mu(a_i \otimes b_j)).
\end{equation}
The pairings $C_k$ and $C'_k$ are related via 
\begin{align*}
C'_k(q(a_1\otimes \dots \otimes a_k) \otimes q(b_1 \otimes \dots \otimes b_k)) &=  C_k(L(a_1 \wedge \dots \wedge a_k) \otimes (b_1 \otimes \dots \otimes b_k))\\
	& = C_k((a_1 \otimes \dots \otimes a_k) \otimes L(b_1 \wedge \dots \wedge b_k))\\
	& = \frac{1}{k!} C_k(L(a_1\wedge \dots \wedge a_k) \otimes L(b_1 \wedge \dots \wedge b_k)).
\end{align*}

The map $C'_k: \wedge^{2k} H_1 \to \Q$ is $\Sp(2g, \Q)$-equivariant (with respect to the trivial action on $\Q$), and it is a standard fact from representation theory that the invariant space $(\wedge^{2k} H_1)^{\Sp(2g,\Q)} \cong \Q$, so that up to scalars, $C'_k$ is the {\em only} such map.

\subsection{The Gysin homomorphism}
In this subsection, we collect some basic information on the Gysin homomorphism. The following proposition, while not treating the absolutely most general case, will suffice for our purposes.

\begin{proposition}[Gysin basics]\label{proposition:gysin}
Suppose that $\pi: E \to B$ is a fibration with $F^n$ a closed oriented $n$-manifold; let $\iota: F \to E$ denote the inclusion of a fiber. Let $M$ be a local system on $B$, determining by pullback a local system (also denoted $M$) on $E$, and restricting to a constant system of coefficients on $F$.
\begin{enumerate}[(i)]
\item \label{item:existence} There are homomorphisms 
\[
\pi_!: H^*(E; M) \to H^{*-n}(B; M)
\]
and 
\[
\pi^!: H_*(B; M) \to H_{*+n}(E; M),
\]
called {\em Gysin homomorphisms}. For $u \in H^n(E; M)$, the Gysin homomorphism simplifies to
\[
\pi_!(u) = \pair{\iota^*(u), [F]},
\]
where $[F] \in H_n(F)$ denotes the fundamental class.

\item \label{item:coeffchange} If $N$ is another local system on $B$ and $f: M \to N$ is a map of local systems, then $f_*$ and $\pi_!$ commute.

\item \label{item:pushpull} Let $u \in H^i(E; M)$ and $v \in H^j(B; N)$ be given. Then there is an equality of elements in $H^{i+j-n}(B; M \otimes N)$
\[
\pi_!(u  \pi^*(v)) = \pi_!(u)  v.
\]

\item \label{item:adjunction} If $u \in H^i(E;M)$ and $x \in H_i(B; N)$ are given, there is an adjunction formula
\[
\pair{\pi_!(u), x} = \pair{u, \pi^!(x)}
\]
of elements in $M \otimes_{\pi_1 B} N$.

\end{enumerate}
\end{proposition}

\section{Twisted MMM classes}\label{section:MMM}
In this section, we review the theory of twisted MMM classes, drawing on the work of Kawazumi and Morita in \cite{km}. As above, let $\Mod_g$ denote the mapping class group of a closed surface, and let $\Mod_{g,*}$ denote the mapping class group of a closed surface with a marked point. There is the projection $\pi: \Mod_{g,*} \to \Mod_g$ giving rise to the Birman exact sequence
\[
\xymatrix{
1 \ar[r] &\pi_1(\Sigma_g) \ar[r]^{\iota}& \Mod_{g,*} \ar[r]^{\pi} &\Mod_g \ar[r] &1.
}
\]
Form the fiber product $\overline{\Mod}_{g,*}$ via the diagram
\[
\xymatrix{
\overline{\Mod}_{g,*} \ar[r]^{\bar \pi} \ar[d]_{\pi} 	&  \Mod_{g,*} \ar[d]\\
\Mod_{g,*} \ar[r]	 \ar@/_/_{\sigma}[u]						& \Mod_g
}
\]
The section $\sigma: \Mod_{g,*} \to \overline{\Mod}_{g,*}$ is given by $\sigma(\phi) = (\phi, \phi)$. There is an isomorphism
\[
\overline{\Mod}_{g,*} \cong   \pi_1(\Sigma_g) \rtimes\Mod_{g,*}
\]
via
\[
(\phi, \psi) \mapsto (\psi \phi^{-1}, \phi).
\]
Under this isomorphism, $\sigma$ is given by $\sigma(\phi) = (1, \phi)$. This semi-direct product decomposition gives rise to a cocycle $k_0 \in Z^1(\overline{\Mod}_{g,*}, H_1)$ via 
\[
k_0((x, \phi)) = [x]. 
\]
By an abuse of notation we will also use $k_0$ to denote the associated element of $H^1(\overline{\Mod}_{g,*}; H_1)$. By construction, $\iota^* k_0 = \id \in H^1(\pi_1 \Sigma_g; H_1)$, and it is also clear that $\sigma^*(k_0) = 0$. 

Let $e \in H^2(\Mod_{g,*})$ denote the Euler class of the vertical tangent bundle. For convenience, let $\bar e \in H^2(\overline{\Mod}_{g,*})$ denote $\bar \pi^*(e)$. The twisted MMM classes defined below were introduced by Kawazumi in \cite{kawazumi}.

\begin{definition}[Twisted MMM classes]\label{definition:mij}
Let $i,j \ge 0$. The {\em twisted MMM class} $m_{ij} \in H^{2i +j -2}(\Mod_{g,*}; H_1^{\otimes k})$ is defined as
\[
m_{ij} = \pi_!(\bar e^i k_0^j).
\]
For $j = 0$, this definition specializes to $m_{i,0} = \pi_!(\bar e^i) = e_{i-1}$, the $(i-1)^{st}$ (classical) MMM class.
\end{definition}

\begin{remark}\label{remark:kfactorial}
Via the graded-commutativity of the cup product, the class $k_0^j \in H^j(\overline{\Mod}_{g,*}; H_1^{\otimes j})$ in fact is valued in the subspace $L(\wedge^j H_1)$, and the same is therefore true of $m_{ij}$. In accordance with our convention that $\wedge^j H_1$ is a quotient of $H_1^{\otimes j}$, we will avoid writing $m_{ij} \in H^{2i + j - 2}(\Mod_{g,*}; \wedge^j H_1)$.
\end{remark}

The formulas at the heart of the present paper are best expressed using a sort of ``interior product''. It will be convenient to first introduce the following piece of notation.
\begin{definition}\label{definition:interlace}
Let $i \ge 2j$ be given. Let $T_{i,j} \in \operatorname{End}{H_1^{\otimes i}}$ be the automorphism induced by permuting the factors via the permutation $f_{ij} \in S_{i}$ given by
\[
f_{ij}(k) = \begin{cases} 
		2k-1 			& k \le j\\
		2(k-j)			& j + 1 \le k \le 2j\\
		k			& k > 2j
	\end{cases}
\]
The effect of $f_{ij}$ is to ``interlace'' the first $2j$ factors, making the $k^{th}$ factor adjacent to the $(k+j)^{th}$ factor. $f_{ij}$ factors as a composition of $j-1 \choose 2$ transpositions of adjacent factors. When $i = 2j$, the notation will be abbreviated to $T_j: = T_{2j,j}$.
\end{definition}

\begin{definition}\label{definition:intprod}
Let $\alpha \in H^m(\Mod_{g,*}; H_1^{\otimes n})$ and $x_i \in H^{d_i}(\Mod_{g,*}; H_1)$ be given for $1 \le i \le k \le n$. Define the class
\[
\alpha \lrcorner (x_1, \dots, x_k) \in H^{m + \sum d_i}(\Mod_{g,*}; H_1^{\otimes n-k})
\]
by the formula
\[
\alpha \lrcorner (x_1, \dots, x_k) = (( \mu^{\otimes k} \otimes \id^{\otimes n-k}) \circ T_{n+k, k})_* (x_1 \dots x_k\ \alpha).
\]
This formula can be equivalently expressed using $C_k$:
\[
\alpha \lrcorner (x_1, \dots, x_k) =  (C_{k} \otimes \id^{\otimes n-k})_*(x_1 \dots x_k\ \alpha).
\]
\end{definition}

Let $f: \Pi \to \Mod_g$ be a homomorphism from a group $\Pi$ to the mapping class group. The fiber product $\Pi_* = \Pi \times_{\Mod_g} \Mod_{g,*}$ admits an extension of groups
\begin{equation}\label{equation:piextension}
\xymatrix{
1 \ar[r] &\pi_1(\Sigma_g) \ar[r]^{\iota}& \Pi_* \ar[r]^{\pi} &\Pi \ar[r] &1.
}
\end{equation}

The following proposition gives a canonical splitting on $H^*(\Pi_*)$. It appears as \cite[Proposition 5.2]{km}.

\begin{proposition}[Kawazumi-Morita]\label{proposition:km52}
Suppose that there exists a cohomology class $\theta \in H^2(\Pi_*)$ such that 
\[
\pi_!(\theta) = \pair{\iota^* \theta, [\Sigma_g]} = 1 \in H^0(\Pi).
\]
Let
\[
\theta' = \theta - \pi^* \pi_!(\theta^2)
\]
which also satisfies $\pi_!(\theta') =1$. The following statements hold:
\begin{enumerate}[(i)]
\item \label{item:splitting} For any $\Q \Pi$-module $M$, the Lyndon-Hochschild-Serre spectral sequence of the extension (\ref{equation:piextension}) collapses at the $E_2$-term, and the cohomology group $H^*(\Pi_*; M)$ naturally decomposes as
\[
H^*(\Pi_*; M) \cong H^{*-2}(\Pi; M) \oplus H^{*-1}(\Pi; H_1\otimes M) \oplus H^*(\Pi;M).
\]

\item \label{item:chi}There exists a unique element $\chi \in H^1(\Pi_*; H_1)$ satisfying
\[
\iota^*\chi = \id \in H^1(\pi_1(\Sigma_g); H_1), \quad \mbox{and}\ \  \pi_!(\theta \chi) = \pi_!(\theta' \chi) = 0.
\]

\item \label{item:epsilon}The homomorphism $\epsilon: H^{*-1}(\Pi; H_1\otimes M) \to H^*(\Pi_*;M)$ given by
\begin{equation}\label{equation:epsilondef}
\epsilon(v) = (\mu\otimes \id_M)_*(\pi^* v\ \chi) \qquad (v \in H^{*-1}(\Pi; H_1\otimes M))
\end{equation}
is a left inverse of the edge homomorphism $\pi_{\sharp}: \ker \pi_! \to E_\infty^{*-1,1} = H^{*-1}(\Pi; H_1\otimes M)$.

\item \label{item:splitformula} Explicitly, for any $u \in H^*(\Pi_*; M)$:
\begin{equation}\label{equation:splitformula}
u = \theta' \pi^* \pi_!(u) - \mu_*( \pi^* \pi_!(u\chi)\ \chi) + \pi^*\pi_!(\theta u).
\end{equation}
\end{enumerate}
\end{proposition}

\begin{remark}\label{remark:universal} The primary case of interest will be the ``universal'' one, taking $\Pi = \Mod_{g,*}$ and $\Pi_* = \overline{\Mod}_{g,*}$. In \cite{moritajacobians}, Morita constructs a class $\nu \in H^2(\overline{\Mod}_{g,*})$ satisfying the properties of $\theta$ listed in Proposition \ref{proposition:km52}. Letting $\chi_\nu$ denote the element $\chi$ associated to $\nu$ given by (\ref{item:chi}) of Proposition \ref{proposition:km52}, Kawazumi-Morita show in \cite{km} that $\chi_\nu = k_0$. 
\end{remark}

As was established by Kawazumi-Morita, the class $\nu \in H^2(\overline{\Mod}_{g,*})$ satisfies certain additional useful formulae; in essence, it behaves like a ``Thom class'' for surface bundles with section. These results are taken from \cite[Theorem 5.1]{km}.

\begin{theorem}[Kawazumi-Morita] \label{theorem:km51}
There is a class $\nu \in H^2(\overline{\Mod}_{g,*})$ satisfying the following properties.
\begin{enumerate}[(i)]
\item $\pi_!\nu = 1$.
\item For any $u \in H^*(\overline{\Mod}_{g,*}; M)$, there is an equality
\[
\nu u = \nu \pi^* \sigma^* u.
\]
Consequently,
\begin{equation}\label{equation:nuu}
\pi_!(\nu u) = \sigma^* u.
\end{equation}
\item\label{item:nue} $\pi_!(\nu^2) = \sigma^* \nu = e$.
\end{enumerate}
\end{theorem}

The following lemma gives a useful alternative characterization of $\im \epsilon$.

\begin{lemma}\label{lemma:epschar}
For all $* \ge 1$, there is an equality 
\[
\im \epsilon = \ker \pi_! \cap \ker \sigma^* 
\]
of subspaces of $H^*(\overline{\Mod}_{g,*})$. 
\end{lemma}

\begin{proof}
The containment $\im \epsilon \subset \ker \pi_!$ follows from the calculation
\begin{align*}
\pi_!(\mu_*(\pi^*u\ k_0)) &= \mu_*(\pi_!(\pi^* u\ k_0))\\
	&= \mu_*(u\pi_!(k_0))\\
	&= 0,
\end{align*}
with the equality $\pi_!(k_0) = 0$ holding for degree reasons. 

To establish the containment $\im \epsilon \subset \ker \sigma^*$, recall the formula (\ref{equation:nuu}). Applied to $u= \mu_*(k_0\ \pi^* u) \in \im \epsilon$, the formula gives
\begin{align*}
\sigma^*v &= \pi_!(\nu\ \mu_*(\pi^* u\ k_0))\\
	&= \pi_!(\mu_*(\nu\ \pi^* u\ k_0))\\
	&= \mu_*(u\ \pi_!(\nu k_0))\\
	&= 0,
\end{align*}
with the equality $ \mu_*(u\ \pi_!(\nu k_0)) = 0$ coming from Proposition \ref{proposition:km52}.\ref{item:chi}. 

The reverse containment is a consequence of the explicit form of the splitting on $H^*(\overline{\Mod}_{g,*})$ given by Proposition \ref{proposition:km52}.\ref{item:splitformula}. If $u \in \ker \pi_! \cap \ker \sigma^*$, then the first and third components in this splitting vanish (recalling that $\pi_!(\nu u) = \sigma^*u$), and so $u \in \im \epsilon$ as desired. \end{proof}

\section{Proof of Theorem \ref{theorem:A}}

The first part of Theorem \ref{theorem:A} asserts the existence of a splitting on $H^*(E)$. This is precisely the content of Proposition \ref{proposition:km52}.\ref{item:splitting}. It remains to establish the formulas for the components given in (\ref{equation:comp1}, \ref{equation:comp2}, \ref{equation:comp3}). 

Per Proposition \ref{proposition:km52}.\ref{item:splitformula}, the $H^{D-2}(B)$-component of $\epsilon(x_1)\dots \epsilon(x_k)$ is given by $\pi_!(\epsilon(x_1)\dots \epsilon(x_k))$. Consider the element
\[
\pi^*(x_1  \dots  x_k)k_0^k \in H^D(E; H_1^{\otimes 2k}).
\]
Recall the interlacing operator $T_{k}$ of Definition \ref{definition:interlace}.
As an automorphism of $H_1^{\otimes 2k}$, it is the composition of $k-1 \choose 2$ transpositions of adjacent factors. Via the graded-commutativity of the cup product,
\[
T_{k,*}( \pi^*(x_1  \dots  x_k)k_0^k) = (-1)^{\gamma} ( \pi^*x_1\ k_0)  \dots  ( \pi^* x_k\ k_0),
\]
where 
\begin{equation}\label{equation:gamma}
\gamma =\sum_{i = 1}^{k-1} (k-i)(d_i - 1).
\end{equation}

From the definition of $\epsilon$ given in Proposition \ref{proposition:km52}.\ref{item:epsilon}, 
\[
\epsilon(x_i) = \mu_*( \pi^* x_i\ k_0).
\]
It follows that
\[
(\mu^{\otimes k} \circ T_k)_*(\pi^*(x_1  \dots  x_k)k_0^k) = (-1)^\gamma \epsilon(x_1)  \dots  \epsilon(x_k). 
\]
Via the commutativity of $(\mu^{\otimes k} \circ T_k)_*$ with $\pi_!$ (Proposition \ref{proposition:gysin}.\ref{item:coeffchange}),
\begin{align*}
\pi_!(\epsilon(x_1)  \dots  \epsilon(x_k)) &= (-1)^\gamma (\mu^{\otimes k} \circ T_k)_* (\pi_!(\pi^*(x_1  \dots  x_k)k_0^k))\\
	&= (-1)^\gamma(\mu^{\otimes k} \circ T_k)_*(x_1 \dots x_k\ m_{0,k})\\
	& = (-1)^\gamma m_{0,k}\lrcorner(x_1, \dots , x_k)
\end{align*}
with the penultimate equality holding as a consequence of the property (\ref{proposition:gysin}.\ref{item:pushpull}) of the Gysin homomorphism and the definition of $m_{0,k}$. This establishes (\ref{equation:comp1}).

Per Proposition \ref{proposition:km52}.\ref{item:splitformula}, the $H^{D-1}(B; H_1)$-component of $\epsilon(x_1)\dots \epsilon(x_k)$ is given by 
\[
-\mu_*(\pi^*\pi_!(\epsilon(x_1) \dots \epsilon(x_k)k_0)k_0) = -\epsilon(\pi_!(\epsilon(x_1)\dots\epsilon(x_k)k_0))
\]
Arguing as in the previous paragraph,
\[
\pi_!(\epsilon(x_1) \dots \epsilon(x_k)k_0) = (-1)^\gamma m_{0,k+1}\lrcorner(x_1, \dots, x_k).
\]
(\ref{equation:comp2}) follows.

It remains to show that the $H^D(B)$-component of $\epsilon(x_1) \dots \epsilon(x_k)$ is $0$. From Proposition \ref{proposition:km52}.\ref{item:splitformula}, this amounts to showing that
\[
\pi_!(\nu \epsilon(x_1) \dots \epsilon(x_k)) = 0.
\]
From (\ref{equation:nuu}) and Lemma \ref{lemma:epschar},
\[
\pi_!(\nu \epsilon(x_1) \dots \epsilon(x_k)) = \sigma^*(\epsilon(x_1) \dots \epsilon(x_k)) = 0.
\]
This establishes (\ref{equation:comp3}). \qed

\section{The restriction of $m_{0,k}$ to $\mathcal I_{g,*}$}\label{section:thmB}
We begin this section with a review of the construction of the higher Johnson invariants. Let $B$ be a paracompact Hausdorff space equipped with a distinguished class $[B] \in H_k(B)$. As the notation suggests, a primary case of interest will be when $B$ is a closed oriented $k$-manifold. Let $f: B \to K(\I_{g,*},1)$ be a map classifying a surface bundle $\pi: E \to B$. Then $f_*([B])$ determines an element of $H_k(K(\I_{g,*},1))$. The space $K(\I_{g,*},1)$ is the base space for a ``universal surface bundle with Torelli monodromy''; i.e. there is a space denoted $K(\overline{\I}_{g,*}, 1)$ and a map $\pi: K(\overline{\I}_{g,*}, 1) \to K(\I_{g,*},1)$ giving $K(\overline{\I}_{g,*}, 1)$ the structure of a $\Sigma_g$-bundle over $K(\I_{g,*},1)$. The total space $E$ therefore determines a $k+2$-cycle 
\[
[E] = \pi^!f_*[B] \in H_{k+2}(\overline{\I}_{g,*}).
\]

By hypothesis, the monodromy representation $\rho: \pi_1(B) \to \I_{g,*}$ is valued in $\I_{g,*}$, so that $H^0(B; H_1(\Sigma_g,\Z)) \cong H_1(\Sigma_g; \Z)$, and there is a section $\sigma: B \to E$. Let $\Jac(E) \to B$ be the $T^{2g}$-bundle obtained by replacing each fiber $\pi^{-1}(b)$ of $E \to B$ with its Jacobian $\Jac(\pi^{-1}(b)) = H_1(\Sigma_g; \R) / H_1(\Sigma_g; \Z)$. The section $\sigma$ endows each fiber $\pi^{-1}(b)$ with a basepoint $\sigma(b)$; consequently there is a fiberwise embedding
\[
J: E \to \Jac(E).
\]
It follows from the equality
\[
H^0(B; H_1(\Sigma_g;\Z)) \cong H_1(\Sigma_g; \Z)
\]
that $\Jac(E) \cong B \times T^{2g}$ is a trivial bundle, so that there is a projection map $p: \Jac(E) \to T^{2g}$.

\begin{definition}[Higher Johnson invariants]\label{definition:tauk}
With notation as above, the {\em $k^{th}$ higher Johnson invariant} $\tau_k(B) \in \wedge^{k+2} H_1$ is the element
\[
p_*J_*[E] \in H_{k+2}(T^{2g}) \cong \wedge^{k+2} H_1.
\]

It is clear from the constructions that if $B, B'$ are homologous $k$-cycles in $K(\I_{g,*},1)$, then $\tau_k(B) = \tau_k(B')$ and that $\tau_k$ is additive. Consequently, $\tau_k$ descends to a homomorphism
\[
\tau_k: H_{k}(\I_{g,*}) \to \wedge^{k+2} H_1;
\]
in view of the Universal Coefficient Theorem, this is equivalent to the description
\[
\tau_k \in H^{k}(\I_{g,*}; \wedge^{k+2} H_1).
\]
\end{definition}

\noindent{\it Proof of Theorem \ref{theorem:B}.} The proof will proceed in two steps. The first step is to understand the relationship between $\tau_{k-2}$ and the structure of the cup product form $\wedge^k H^1(E) \to H^k(E) \to \Q$ (this last map is obtained by the pairing $\alpha \mapsto \pair{\alpha, [E]}$). Once this is established, the second step is to compare this to the relationship between $m_{0,k}$ and the cup product form established by Theorem \ref{theorem:A}.

\para{Step 1: The higher Johnson invariants record the cup product form}
\begin{proposition}\label{proposition:johnsoncup}
Let $f: B \to K(\I_{g,*},1)$ determine a $k-2$-cycle $[B]$ in $K(\I_{g,*},1)$ and let $[E]$ be the associated $k$-cycle in $K(\overline{\I}_{g,*},1)$. Let $\epsilon: H^{*-1}(B; H_1) \to H^*(E)$ be the map defined in Proposition \ref{proposition:km52}.\ref{item:epsilon}, and let $a_1, \dots, a_k \in H_1$ be given. Then
\[
\pair{\epsilon(a_1)  \dots  \epsilon(a_k), [E]} = (-1)^k C'_k((a_1 \wedge \dots \wedge a_k)\otimes \tau_{k-2}[B]).
\]
\end{proposition}

\begin{proof}
The symplectic pairing $\mu: H_1^{\otimes 2} \to \Q$ induces an isomorphism $\cdot^\vee: H_1 \to H^1$ given by $w^\vee(u) = \mu(u \otimes w)$. By pullback, any $w \in H_1 \cong H_1(T^{2g})$ determines the class $J^*p^*w^\vee \in H^1(E)$.

We claim that there is an equality for any $w \in H_1$,
\[
\epsilon(w) = J^* p^*w^\vee.
\]
The first step is to show that $\im (J^*p^*) \subseteq \im \epsilon$. This will follow from Lemma \ref{lemma:epschar}. For degree reasons, $\pi_!(J^*p^*w^\vee) = 0$. It remains to show that $\sigma^*(J^* p^*w^\vee) = 0$. By construction, $p \circ J \circ \sigma: B \to T^{2g}$ is the constant map sending $B$ to $0 \in T^{2g}$; the result follows.

Given $w \in H_1$, we have shown that there is some $v \in H_1 = H^0(B; H_1)$ such that $J^* p^*w^\vee = \epsilon(v)$. It remains to show that $v = w$. Let $\iota: \Sigma_g \to E$ be the inclusion of a fiber. The composition $p \circ J \circ \iota: \Sigma_g \to T^{2g}$ coincides with the Jacobian mapping. Consequently, $\iota^*(J^* p^*w^\vee) = w^\vee$. 

On the other hand,
\[
\iota^*(\epsilon(v)) = \iota^*(\mu_*(\pi^*v\ k_0)) = \mu_*(\iota^*(\pi^* v\ k_0)).
\]
Let $u \in H_1$ be arbitrary. Then
\[
\pair{\mu_*(\iota^*(\pi^* v\ k_0)), u} = \mu(\pair{\iota^*(\pi^* v\ k_0), u}).
\]
As $\iota^* k_0 = \id$, the above formula simplifies to
\[
\mu(\pair{\iota^*(\pi^* v\ k_0), u}) = -\mu(v \otimes u) = v^\vee (u).
\]
Consequently, $w^\vee = v^\vee$, from which the equality $w = v$ follows. 

From the above, there is an expression
\begin{align*}
\pair{\epsilon(a_1)  \dots  \epsilon(a_k), [E]} &= \pair{J^* p^*(a_1^\vee  \dots  a_k^\vee), [E]} \\
	&= \pair{a_1^\vee  \dots  a_k^\vee, p_* J_*[E]}\\
	&= \pair{a_1^\vee  \dots  a_k^\vee, \tau_{k-2}[B]}.
\end{align*}
Under the isomorphisms $H_k(T^{2g}) \cong \wedge^k H_1$ and $H^k(T^{2g}) \cong \wedge^k H^1$, the evaluation pairing $H^k(T^{2g}) \otimes H_k(T^{2g}) \to \Q$ is mapped to the pairing
\begin{equation}\label{equation:detpair}
(\alpha_1 \wedge \dots \wedge \alpha_k ) \otimes (a_1 \wedge \dots \wedge a_k) \mapsto \det(\alpha_i(a_j)). 
\end{equation}
Under the embedding 
\[
\wedge^k(\cdot^\vee)^{-1} \otimes \id: \wedge^k H^1 \otimes \wedge^k H_1 \to (\wedge^k H_1)^{\otimes 2},
\]
the pairing (\ref{equation:detpair}) corresponds to $(-1)^kC'_k$. Consequently,
\[
 \pair{a_1^\vee  \dots  a_k^\vee, \tau_{k-2}[B]} = (-1)^kC'_k((a_1 \wedge \dots \wedge a_k)\otimes \tau_{k-2}[B])
\]
as was to be shown.
\end{proof}

\para{Step 2: Comparison with $m_{0,k}$} Suppose that $B$ determines a $(k-2)$-cycle in $K(\I_{g,*},1)$. We must show that
\[
q(\pair{m_{0,k},[B]}) = (-1)^kk! \ \tau_{k-2}[B],
\]
where, as in Section \ref{subsection:multi}, the map $q: H_1^{\otimes k} \to \wedge^k H_1$ is the projection. As the pairing $C'_k: (\wedge^k H_1)^{\otimes 2} \to \Q$ of (\ref{equation:C'def}) is nondegenerate, it suffices to show the equality of the forms:
\[
a_1 \wedge \dots \wedge a_k \mapsto (-1)^kC'_k((a_1 \wedge \dots \wedge a_k) \otimes \tau_{k-2}[B])
\] and \[
a_1 \wedge \dots \wedge a_k \mapsto \frac{(1}{k!} C'_k((a_1 \wedge \dots \wedge a_k)\otimes q(\pair{m_{0,k},[B]})).
\]
Proposition \ref{proposition:johnsoncup} asserts that for $a_1,\dots, a_k \in H_1$, there is an equality
\begin{align*}
\pair{\epsilon(a_1)  \dots  \epsilon(a_k), [E]} &= (-1)^kC'_k((a_1 \wedge \dots \wedge a_k)\otimes \tau_{k-2}[B])\\
\end{align*}
Proposition \ref{proposition:gysin}.\ref{item:adjunction} implies:
\begin{align*}
\pair{\epsilon(a_1)  \dots  \epsilon(a_k), [E]} &= \pair{\epsilon(a_1)  \dots  \epsilon(a_k), \pi^![B]}\\ &= \pair{\pi_!(\epsilon(a_1)  \dots  \epsilon(a_k)), [B]}.
\end{align*}
Theorem \ref{theorem:A} implies:
\begin{align*}
\pair{\pi_!(\epsilon(a_1)  \dots  \epsilon(a_k)), [B]} &= \pair{m_{0,k} \lrcorner(a_1, \dots, a_k), [B]}\\
	&= \pair{C_{k,*}(a_1 \dots a_k\ m_{0,k}), [B]}\\
	&= C_k(\pair{a_1 \dots a_k\ m_{0,k}, [B]})\\
	&= C_k((a_1 \otimes \dots \otimes a_k)\otimes \pair{ m_{0,k}, [B]})\\
	&= C_k((a_1 \otimes \dots \otimes a_k) \otimes \pair{ m_{0,k}, [B]})
\end{align*}
(here $\gamma = 0$ as each $d_i = 1$). As $m_{0,k} \in H^{k-2}(\Mod_{g,*}; L(\wedge^k H_1))$, there is an expression of the form
\[
\pair{m_{0,k}, [B]} = L(\zeta)
\] 
for some $\zeta \in \wedge^k H_1$. It follows that $q(\pair{m_{0,k}, [B]}) = k! \zeta$. The results of Section \ref{subsection:multi} imply:
\begin{align*}
C_k((a_1 \otimes \dots \otimes a_k)\otimes \pair{ m_{0,k}, [B]}) & = \frac{1}{k!} C'_k((a_1 \wedge \dots \wedge a_k) \otimes q(\pair{m_{0,k}, [B]})).
\end{align*}

The result follows. \qed

\section{Relation to MMM classes: Theorem \ref{theorem:C}}
This section is devoted to the proof of Theorem \ref{theorem:C}. This will be divided into two steps. The first step is to establish a contraction formula for $\mu_{0,2n}$. The second step will be to relate this to the representation theory of $\Sp(2g, \Q)$.  

\para{Step 1: Contraction formula}
The first step is to calculate $\mu^{\otimes n}_*(m_{0,2n}) \in H^{2n-2}(\Mod_{g,*})$. We claim that the following formula holds:
\begin{equation}\label{equation:m0n}
\mu^{\otimes n}_*(m_{0,2n}) = (-1)^{n-1}2^{n}e^{n-1} + (-1)^n \sum_{i = 1}^n {n \choose i} e^{n-i} e_{i-1}.
\end{equation}
By convention, $e_0 = 2 - 2g \in H^0(\Mod_{g,*})$. 

According to \cite[Theorem 6.1]{km} there is an expression for $\mu_*(k_0^2) \in H^2(\overline{\Mod}_{g,*})$ of the form
\[
\mu_*(k_0^2) = 2 \nu - e - \bar e.
\]
Therefore,
\[
\mu_*^{\otimes n}(k_0^{2n}) = (2 \nu - e - \bar e)^n.
\]
It follows from Proposition \ref{proposition:gysin}.\ref{item:coeffchange} that
\[
\mu_*^{\otimes n} (m_{0,2n}) = \pi_!((2 \nu - e - \bar e)^n).
\]

Recall that $e \in H^2(\overline{\Mod}_{g,*})$ is defined as $\pi^*(e)$ for $e \in H^2(\Mod_{g,*})$, and $\bar e$ is defined as $\bar \pi ^*(e), e \in H^2(\Mod_{g,*})$. Equation (\ref{equation:nuu}) of Theorem \ref{theorem:km51} asserts that $\pi_!(\nu  x) = \sigma^*(x)$. The composition $\bar \pi \circ \sigma = \id$, and so $\sigma^*(e) = \sigma^*(\bar e) = e$. Theorem \ref{theorem:km51}.\ref{item:nue} implies that $\sigma^*(\nu) = e$.

Expand $(2 \nu  - e - \bar e)^n$ as 
\[
(2 \nu  - e - \bar e)^n = 2\nu (2 \nu - e - \bar e)^{n-1} -(e + \bar e) (2 \nu - e - \bar e)^{n-1}.
\]
For $n \ge 2$, the above discussion shows that $\pi_!(2\nu (2 \nu - e - \bar e)^{n-1}) = 2 \sigma^*(2 \nu - e - \bar e)^n =  0$. It follows that
\[
\pi_!((2 \nu  - e - \bar e)^n) = - \pi_!((e + \bar e) (2 \nu - e - \bar e)^{n-1}), 
\]
and that in general, for $j \le n-2$,
\[
\pi_!((e + \bar e)^j(2 \nu  - e - \bar e)^{n-j}) = - \pi_!((e + \bar e)^{j+1} (2 \nu - e - \bar e)^{n - j - 1}).
\]
Applying this formula repeatedly,
\begin{align*}
\pi_!((2 \nu - e - \bar e)^n) &= (-1)^{n-1} \pi_!((e + \bar e)^{n-1}(2 \nu - e - \bar e))\\
	&= (-1)^{n-1} \pi_!(2 \nu (e + \bar e)^{n-1}) + (-1)^{n} \pi_!((e+ \bar e)^n)\\
	&= (-1)^{n-1} 2^n e^{n-1} + (-1)^{n} \sum_{i = 1}^n {n \choose i} e^{n-i} e_{i-1}.
\end{align*}
In the last equality, we have applied Proposition \ref{proposition:gysin}.\ref{item:pushpull}, recalling that $e$ is the pullback $\pi^*(e), e \in H^2(\Mod_{g,*})$. 

\para{Step 2: Contractions in symplectic representation theory} As the restriction of $e$ to $H^2(\I_g^1)$ is zero, Step 1 implies that the pullback of $e_i$ to $H^{2i}(\I_g^1)$ is zero if and only if $\mu_*^{\otimes {i+1}}(m_{2i+2})$ vanishes in $H^{2i}(\I_g^1)$. Theorem \ref{theorem:B} implies that this is in turn equivalent to the vanishing of $\mu_*^{\otimes i+1}(\tau_{2i})$.  

 In the notation of Section \ref{subsection:multi}, there is a decomposition 
\[
\wedge^{2i+2} H_1 = V(\lambda_{2i+2}) \oplus V(\lambda_{2i}) \oplus \dots \oplus V(\lambda_0).
\]
Treating $\wedge^{2i+2} H_1$ as a subspace of $(H_1^{\otimes 2})^{\otimes i+1}$, the contraction $\mu^{\otimes i+1}$ is a map of $\Sp(2g, \Q)$-representations projecting onto $V(\lambda_0) \cong \Q$. Viewed as an element of $\Hom(H_{2i}(\I_g^1), \Q)$, the class $\mu_*^{\otimes i+1} (\tau_{2i})$ is therefore nonzero if and only if 
\[
V(\lambda_0) \le \im(\tau_{2i}).
\]
This completes the proof of Theorem \ref{theorem:C}. \qed

\section{Applications to surface bundles}\label{section:applications}
In this last section, we turn from a study of global cohomology classes on $\Mod_g$ and $\I_g$ in favor of a study of $H^*(E)$ for $\pi:E \to B$ a particular $\Sigma_g$-bundle over a paracompact Hausdorff space $B$. The particular bundles under consideration will have an additional constraint on their monodromy representations, namely that $\rho: \pi_1 B \to \K_{g,*}$ is valued in the {\em Johnson kernel} $\K_{g,*} = \ker(\tau: \I_{g,*} \to \wedge^3 H_1)$. It is a deep fact due to Johnson \cite{johnsonii} that equivalently,
\begin{equation}\label{equation:johnsonkernel}
\K_{g,*} = \pair{T_\gamma \mid \gamma \mbox{ separating}},
\end{equation}
i.e. that the Johnson kernel is the group generated by all Dehn twists about {\em separating} simple closed curves. There is an analogous definition of $\K_{g} \le \Mod_g$ and a statement analogous to (\ref{equation:johnsonkernel}). \\

\noindent{\em Proof of Theorem \ref{theorem:jkbundles}:} The method will be to exploit Theorem \ref{theorem:A}. We will show that under the splitting of graded vector spaces
\[
H^*(E) \cong H^*(B) \otimes H^*(\Sigma_g),
\]
the multiplication on $H^*(E)$ induced by the cup product agrees with the ring structure on $H^*(B) \otimes H^*(\Sigma_g)$ induced by the cup products on $B$ and $\Sigma_g$. This will be accomplished by a separate verification on the six different pairs of subspaces $(H^*(B) \otimes H^{i}(\Sigma_g))\otimes (H^*(B) \otimes H^{j}(\Sigma_g))$ of $H^*(E)^{\otimes 2}$ for $0 \le i \le j \le 2$. 

For the readers convenience we list below the inclusions $F: H^m(B) \otimes H^i(\Sigma_g) \to H^{m+i}(E)$ of Theorem \ref{theorem:A} that will yield the ring isomorphism. We have identified $H^1(\Sigma_g) \cong H_1(\Sigma_g)$ by means of $\mu$. A generator of $H^2(\Sigma_g)$ will be denoted $\omega$. 
\begin{align*}
F(u \otimes 1)  		&= \pi^*u 						&(H^m(B) \otimes H^0(\Sigma_g) \to H^m(E))\\
F(u \otimes x)		&= \mu_*(\pi^*(u \otimes x) k_0) 	&(H^m(B) \otimes H^1(\Sigma_g) \to H^{m+1}(E))\\
F(u \otimes \omega)	&= \pi^*u\ \nu'					&(H^m(B) \otimes H^2(\Sigma_g) \to H^{m+2}(E))
\end{align*}

The table below records the multiplicative structure on $H^*(B) \otimes H^*(\Sigma_g)$ induced by the cup products on $B$ and $\Sigma_g$. Under the identification $H^1(\Sigma_g) \cong H_1(\Sigma_g)$, the cup product is given by $x y = \mu(x,y) \omega$.
\[
\begin{array}{|c|ccccc|}
\hline
				&v \otimes 1	&\ & v \otimes y							&\ & v \otimes \omega\\ \hline
u \otimes 1		&uv \otimes 1	&& uv \otimes y							&& uv \otimes \omega\\ 
u \otimes x		&			&&(-1)^{\abs v} \mu(x,y) uv \otimes \omega	&&	0\\ 
u \otimes \omega	&			&&									&& 0 \\ \hline
\end{array}
\]

Passing the entries in this table through $F$ yields a table of values for $F(ab)$ (for $a,b \in H^*(B) \otimes H^*(\Sigma_g)$):
\[
\begin{array}{|c|ccccc|}
\hline
				&v \otimes 1	&& v \otimes y						&& v \otimes \omega\\ \hline
u \otimes 1		&\pi^*(uv)		&& \mu_*(\pi^*(uv \otimes y) k_0)		&& \pi^*(uv)\ \nu'\\ 
u \otimes x		&			&&(-1)^{\abs v} \mu(x,y) \pi^*(uv)\ \nu'	&&	0\\ 
u \otimes \omega	&			&&								&& 0 \\ \hline
\end{array}
\]

Showing that $F$ is a ring isomorphism reduces to showing that this table matches the table of values for $F(a)F(b)$, given below.
\[
\begin{array}{|c|ccc|}
\hline
							&\pi^*v	& \mu_*(\pi^*(v \otimes y) k_0)							&\pi^*v\ \nu'\\ \hline
\pi^*u						&\pi^*(uv)	& \pi^*u\ \mu_*(\pi^*(v \otimes y) k_0)					& \pi^*(uv)\ \nu' \\ 
\mu_*(\pi^*(u \otimes x) k_0)		&		& \mu_*(\pi^*(u \otimes x) k_0)\ \mu_*(\pi^*(v \otimes y) k_0)	&\mu_*(\pi^*(u \otimes x) k_0)\ \pi^*v\ \nu'\\
\pi^*u\ \nu'						&		&												& \pi^*(uv) (\nu')^2 \\ \hline
\end{array}
\]
The first pair of entries to reconcile is $\mu_*(\pi^*(uv \otimes y) k_0)$ and $ \pi^*u\ \mu_*(\pi^*(v \otimes y) k_0)$. This is essentially immediate. We must next show the equality
\[
(-1)^{\abs v} \mu(x,y) \pi^*(uv)\ \nu' = \mu_*(\pi^*(u \otimes x) k_0)\ \mu_*(\pi^*(v \otimes y) k_0).
\]
Calculating,
\begin{align*}
\mu_*(\pi^*(u \otimes x) k_0)\ \mu_*(\pi^*(v \otimes y) k_0) &= (-1)^{\abs v} C_{2,*} (\pi^*(u \otimes x) \pi^*(v \otimes y) k_0^2)\\
	&= (-1)^{\abs v} \pi^*(uv) C_{2,*}(\pi^*(x\otimes y)\ k_0^2).
\end{align*}
Here, $(x \otimes y)$ is to be interpreted as an element of $H^0(B; H_1^{\otimes 2})$. Clearly the equality will be established if the statement
\[
C_{2,*}(\pi^*(x\otimes y)\ k_0^2) = \mu(x,y) \nu'
\]
is shown to hold. To do this, the components of $C_{2,*}(\pi^*(x\otimes y)\ k_0^2)$ will be computed for the splitting on $H^*(E)$ given by $F$. To compute $\pi_!(C_{2,*}(\pi^*(x\otimes y)\ k_0^2))$, observe that
\begin{align*}
\pi_!(C_{2,*}(\pi^*(x\otimes y)\ k_0^2)) &= C_{2,*}((x \otimes y) \pi_!(k_0^2))\\
	&= C_{2,*}((x \otimes y) \iota^*(k_0^2)) \\
	&= C_{2,*}((x \otimes y) \id^2).
\end{align*}
The last equality holds in light of the fact that $\iota^* k_0 = \id \in H^1(\Sigma_g; H_1)$. From here, an examination of the definition of $C_{2,*}$ shows that $C_{2,*}((x \otimes y) \id^2) = \mu(x,y)$. 

The next step is to compute the $H^*(B; H_1)$-component of $C_{2,*}(\pi^*(x\otimes y)\ k_0^2)$; the goal is to show this is zero. This is computed as follows:
\begin{align*}
\mu_*(\pi^*\pi_!(C_{2,*}(\pi^*(x\otimes y)\ k_0^3)) k_0) &= \mu_*(\pi^*(m_{0,3}\lrcorner(x,y))k_0).
\end{align*}
Theorem \ref{theorem:B} asserts that $m_{0,3} = -6 \tau_1$. Therefore $m_{0,3} = 0$ when restricted to $\K_{g,*}$, showing that the $H^*(B; H_1)$-component of $C_{2,*}(\pi^*(x\otimes y)\ k_0^2)$ is zero as desired.

The final step is to show that 
\[
\pi_!(\nu C_{2,*}(\pi^*(x\otimes y)\ k_0^2)) = 0,
\]
or equivalently that $\sigma^*(C_{2,*}(\pi^*(x\otimes y)\ k_0^2))=0$. This latter expression is divisible by $\sigma^*(k_0) = 0$, and the result follows.\\

To complete the proof of Theorem \ref{theorem:jkbundles}, it remains to show the vanishing of $\mu_*(\pi^*(u \otimes x) k_0)\ \pi^*v\ \nu'$ and of $\pi^*(uv) (\nu')^2$. To show the former, it suffices to show that $k_0 \nu' =0$ when restricted to $\K_{g,*}$. This will be shown by computing the components of $k_0 \nu'$ in the splitting given by $F$. $\pi_!(k_0 \nu') = 0$ is seen to hold immediately by properties of $\nu'$ and $k_0$. It must next be shown that
\begin{equation}\label{equation:pikachu}
\mu_*(\pi^*\pi_!(\nu' k_0^2)k_0) = 0.
\end{equation}
Recall that 
\[
\nu' = \nu - \pi^* \pi_!(\nu^2) = \nu - \pi^* \sigma^*(\nu) = \nu - e.
\]
According to \cite[Theorem 5.1]{moritalinear96}, the Euler class $e \in H^2(\Mod_{g,*})$ is in the image of the pullback $\rho_1^*$, where $\rho_1$ is the map
\[
\rho_1: \Mod_{g,*} \to \frac{1}{2} \wedge^3 H_1 \rtimes \Sp(2g,\Z) 
\]
given by $\rho_1(\phi) = (\tilde k(\phi), \Psi(\phi))$. Restricted to $\I_{g,*}$, the map $\rho_1$ simplifies to the Johnson homomorphism $\tau_1$, and so $\rho_1^*$ has zero image when pulled back to $H^2(\K_{g,*})$. It follows that $e = 0$, and so, when restricted to $\K_{g,*}$, there is an equality $\nu' = \nu$. Therefore, the term $\pi_!(\nu' k_0^2)$ in (\ref{equation:pikachu}) simplifies to $\pi_!(\nu k_0^2) = \sigma^*(k_0^2) = 0$. Likewise, 
\[
\pi_!(\nu \nu' k_0) = \pi_!(\nu^2 k_0) = \sigma^*(\nu k_0) = 0,
\]
and the final component of $\nu' k_0$ is seen to vanish.\\

It remains only to show $\pi^*(uv) (\nu')^2 = 0$, which is obviously implied by showing $(\nu')^2 =0$. As was remarked in the previous step, $\nu' = \nu$ on $\K_{g,*}$. As before, we will show $\nu^2 = 0$ by computing the components of $\nu^2$. The first of these is divisible by the factor
\[
\pi_!(\nu^2) = \sigma^*(\nu) = e = 0,
\]
while the third is
\[
\pi_!(\nu^3) = \sigma^*(\nu^2) = e^2 = 0. 
\]

The remaining step is to show 
\[
\mu_*(\pi^*\pi_!(\nu^2 k_0)k_0) = 0.
\]
This follows from the vanishing $\pi_!(\nu^2 k_0) = 0$ established above. \qed\\

Finally, Theorem \ref{theorem:jkvanish} follows as a corollary.

\noindent{\em Proof of Theorem \ref{theorem:jkvanish}:} Let $f: B \to K(\K_{g,*},1)$ determine a $\Sigma_g$-bundle $\pi: E \to B$ with monodromy contained in $\K_{g,*}$; let $B$ be equipped with the distinguished homology class $[B] \in H_k(B)$. Proposition \ref{proposition:johnsoncup} asserts that for any $a_1, \dots, a_{k+2} \in H_1$, there is an equality
\[
\pair{\epsilon(a_1)  \dots  \epsilon(a_{k+2}), [E]} = (-1)^k C'_k((a_1 \wedge \dots \wedge a_{k+2})\otimes \tau_{k}[B]).
\]
As $C_k'$ is nondenegerate, it suffices to show that $\pair{\epsilon(a_1)  \dots  \epsilon(a_{k+2}), [E]} = 0$ for all $k+2$-tuples $a_1, \dots, a_{k+2} \in H_1$. From Theorem \ref{theorem:jkbundles}, there is an expression
\[
\epsilon(a_1) \epsilon(a_2) = \mu(a_1, a_2) \nu.
\]
Theorem \ref{theorem:jkbundles} also asserts that $\nu\ \epsilon(a_3) = 0$, so that the triple product $\epsilon(a_1) \epsilon(a_2) \epsilon(a_3) = 0$. The result follows. \qed

    	\bibliography{multiplication}{}
	\bibliographystyle{alpha}

\end{document}